\documentclass[11pt]{amsart}
\usepackage{amscd}
\usepackage{amsmath}
\usepackage{amsxtra}
\usepackage{amsfonts}
\usepackage{amssymb}
\usepackage{color}

\newtheorem{theorem}{Theorem}[section]
\newtheorem{corollary}[theorem]{Corollary}
\newtheorem{lemma}[theorem]{Lemma}
\newtheorem{proposition}[theorem]{Proposition}

\theoremstyle{definition}
\newtheorem{definition}[theorem]{Definition}

\theoremstyle{remark}

\renewcommand{\theclaim}{\textup{\theclaim}}

\numberwithin{equation}{section}

\def\openone

{\mathchoice

{\hbox{\upshape \small1\kern-3.3pt\normalsize1}}

{\hbox{\upshape \small1\kern-3.3pt\normalsize1}}

{\hbox{\upshape \tiny1\kern-2.3pt\SMALL1}}

{\hbox{\upshape \Tiny1\kern-2pt\tiny1}}}

\makeatletter

\newbox\ipbox

\newcommand{\diracb}[1]{\left\langle #1\mathrel{\mathchoice

{\setbox\ipbox=\hbox{$\displaystyle \left\langle\mathstrut
#1\right.$}

\vrule height\ht\ipbox width0.25pt depth\dp\ipbox}

{\setbox\ipbox=\hbox{$\textstyle \left\langle\mathstrut
#1\right.$}

\vrule height\ht\ipbox width0.25pt depth\dp\ipbox}

{\setbox\ipbox=\hbox{$\scriptstyle \left\langle\mathstrut
#1\right.$}

\vrule height\ht\ipbox width0.25pt depth\dp\ipbox}

{\setbox\ipbox=\hbox{$\scriptscriptstyle \left\langle\mathstrut
#1\right.$}

\vrule height\ht\ipbox width0.25pt depth\dp\ipbox}

}\right. }

\newcommand{\dirack}[1]{\left. \mathrel{\mathchoice

{\setbox\ipbox=\hbox{$\displaystyle \left.\mathstrut
#1\right\rangle$}

\vrule height\ht\ipbox width0.25pt depth\dp\ipbox}

{\setbox\ipbox=\hbox{$\textstyle \left.\mathstrut
#1\right\rangle$}

\vrule height\ht\ipbox width0.25pt depth\dp\ipbox}

{\setbox\ipbox=\hbox{$\scriptstyle \left.\mathstrut
#1\right\rangle$}

\vrule height\ht\ipbox width0.25pt depth\dp\ipbox}

{\setbox\ipbox=\hbox{$\scriptscriptstyle \left.\mathstrut
#1\right\rangle$}

\vrule height\ht\ipbox width0.25pt depth\dp\ipbox}

} #1\right\rangle}

\newcommand{\beq}{\begin{equation}}
\newcommand{\eeq}{\end{equation}}

\def\blfootnote{\xdef\@thefnmark{}\@footnotetext}

\input xy
\xyoption{all}
\usepackage{amssymb}

\def\R{\mathbb{R}}

\def\-{^{-1}}

\def\Z{\mathbb{Z}}

\begin{document}

\title[Arbitrarily sparse spectra for self-affine spectral measures]{Arbitrarily sparse spectra for self-affine spectral measures}
\author{Lixiang An}
\address{[Lixiang An]School of Mathematics and Statistics,
$\&$ Hubei Key Laboratory of Mathematical Sciences,
Central China Normal University,
Wuhan 430079,
P.R. China.}

 \email{anlixianghai@163.com}

%
%
%
%

\author{Chun-Kit Lai}

\address{[Chun-Kit Lai]Department of Mathematics, San Francisco State University,
1600 Holloway Avenue, San Francisco, CA 94132.}

 \email{cklai@sfsu.edu}

\thanks{The research of Lixiang An is supported by NSFC  grant 11601175.}

\begin{abstract}
Given an expansive matrix $R\in M_d(\Z)$ and a finite set of digit $B$ taken from $ \Z^d/R(\Z^d)$. It was shown previously that if we can find an $L$ such that $(R,B,L)$ forms a Hadamard triple, then the associated fractal self-affine measure generated by $(R,B)$ admits an exponential orthonormal basis of certain frequency set $\Lambda$, and hence it is termed as a spectral measure. In this paper, we show that if $\#B<|\det (R)|$, not only it is spectral, we can also construct arbitrarily sparse spectrum $\Lambda$ in the sense that  its Beurling dimension is zero. 
\end{abstract}

\maketitle

\begin{center}
{\it In memory of Dr. Tian-You Hu\footnote{T.-Y Hu was a dedicated mathematician. He was also a great mentor and a dear friend of the authors of this paper. Academically, he introduced us into the field of spectral measures through his paper \cite{HL2008}. It was sad to hear that he passed away due to Covid-19.}}
\end{center}

\medskip
\section{Introduction}

\subsection{Definitions and main results.} Let $R\in M_{d}(\Z)$ be an expansive matrix (i.e. all of its eigenvalues have modulus strictly greater than 1). Let $B, L\subset\Z^d$ be finite sets of integer vectors  with $q:=\#L=\#B$. We say that the system  $(R, B, L)$ forms a {\it Hadamard triple} if the matrix
$$H=\frac{1}{\sqrt{q}}\left[e^{-2\pi i \langle R^{-1}b, l\rangle}\right]_{b\in B, l\in L}$$
is unitary, i.e., $H^*H=I$.  

\medskip

Given an expansive matrix $R\in M_d(\Z)$ and given  $B\subset \Z^d$. By the result of Hutchinson \cite{Hut81}, we can define the {\it affine iterated function system} (IFS) \{$\tau_b(x)=R^{-1}(x+b), x\in\R^d, b\in B$\} which has a unique compact attractor $T(R,B)$, called {\it self-affine set},  satisfying 
$$T(R, B)=\bigcup_{b\in B}\tau_b(T(R, B)).$$
 The {\it self-affine measure} (with equal weights) is the unique probability measure $\mu=\mu(R, B)$ satisfying
$$\mu(E)=\frac1{q}\sum_{b\in B}\mu(\tau_b^{-1}(E))$$
for all Borel subsets $E$ of $\R^d$. This measure is supported on the attractor $T(R, B)$.

\medskip

In a previous work, Dutkay, Hausserman and the second-named author proved the following theorem (\cite{DHL2019}, see \cite{LW2002} for the proof on $\R^1$).

\begin{theorem}
 Suppose that $(R,B,L)$ forms a Hadamard triple on $\R^d$. Then the self-affine measure $\mu(R,B)$ admits an exponential orthonormal basis $E(\Lambda) = \{e^{2\pi i \lambda\cdot x}:\lambda\in \Lambda\}$ for some countable set $\Lambda\subset {\mathbb R}^d$. 
 \end{theorem}

\medskip

We say that a  Borel probability measure $\mu$ on ${\mathbb R}^d$ is called a {\it spectral measure} if we can find a countable set $\Lambda\subset{\mathbb R}^d$ such that the set of exponential functions $E(\Lambda): = \{e^{2\pi i \lambda \cdot x}:\lambda\in\Lambda\}$ forms an orthonormal basis for $L^2(\mu)$. If such $\Lambda$ exists, then $\Lambda$ is called a {\it spectrum} for $\mu$.  The above theorem said that $\mu(R,B)$ is a spectral measure if $(R,B,L)$ forms a Hadamard triple with some $L\subset{\mathbb Z}^d$.

\medskip

In this paper, We study in more detail about the sparseness of the spectrum as measured by Beurling dimension.

\begin{definition}\label{Beurling_dimension} Denote  $Q^d_h(x)=x+[-h, h]^d$ be the cube centered at $x$ on $\R^d$.
\medskip

(1) Let $\Lambda$ be a discrete subset of ${\mathbb R}^d$. For $r>0$, the {\it upper Beurling density}  corresponding to $r$ (or {\it $r$-Beurling density}) is defined by
$$D_r^+(\Lambda)=\limsup_{h\to\infty}\sup_{x\in{\mathbb R}^d}\frac{\#(\Lambda\cap Q_h^d(x))}{h^r}.$$

\bigskip
(2) The {\it upper Beurling dimension} (or simply the {\it Beurling dimension}) of $\Lambda$ is defined by
$$
\dim^+(\Lambda):=\sup\{r>0: D_r^+(\Lambda)>0\}=\inf\{r>0: D_r^+(\Lambda)<\infty\}.
$$
\end{definition}

\medskip

Our main result is the following. 

\begin{theorem}\label{main-thm}
If $(R, B, L)$ forms a Hadamard triple on $\R^d$ with $\#B<|\det R|$, then the spectral measure $\mu(R, B)$ admits a spectrum $\Lambda$ with Beurling dimension zero. 
\end{theorem}

\medskip

\subsection{Historical overview.} Historically, spectral measure was first studied by Fuglede who introduced the notion of spectral sets and explored its relationship with translational tile \cite{Fug1974}.  The first singularly continuous spectral measure was found by Jorgensen and Pedersen \cite{JP1998}. They showed that  the standard middle-fourth Cantor measures $\mu_4$ is spectral, while the middle-third Cantor measure $\mu_3$ is not spectral. $\mu_4$ is generated by $R = 4$ and $B = \{0,2\}$. $(R,B,L)$ forms a Hadamard triple with $L = \{0,1\}$. Using $L$, they found that  one of the spectra of $\mu_4$ is given by 
$$
\Lambda_0  =  \left\{\sum_{j=0}^{n-1}4^j\epsilon_j: \epsilon_j\in\{0,1\}, n\ge 1\right\}.
$$
This is, however, not the only spectrum for $\mu_4$. We can also see that $(4,\{0,2\}, L_n)$ with $L_n = \{0,5^n\}$  also form Hadamard triples.  Indeed, one can also show that $5^n \Lambda_0$ all are spectra of $\mu_4$. A direct calculation shows that all these spectra have Beurling dimension $\ln 2/\ln 4$.

\medskip

In an attempt to capture the right density condition for the spectra of $\mu_4$,  Dutkay, Han, Sun and Weber \cite{DHSW2011} proposed the notion of Beurling dimension, and they brought this notion from  the study of Gabor pseudo-frame in \cite{CKS2008}.  In \cite{DHSW2011}, the authors showed that all spectra of $\mu_4$ must have Beurling dimension at most $\ln2/\ln 4$ which is the Hausdorff dimension of the attractor.  Under a technical condition on the spectrum $\Lambda$, a spectrum of $\mu$ must have a Beurling dimension $\ln2/\ln4$. It used to be a conjecture that the technical condition can be removed.  However, Dai, He and the second-named author disproved the conjecture by exhibiting  a spectrum of Beurling dimension zero for $\mu
_4$ \cite{DHL2013}. The existence of sparse spectra with Beurling dimension zero is also true for other one-dimensional self-similar measures whose digit sets are consecutive $\{0,1,...,q-1\}$ \footnote{In private communication with Y. Wang and B. Strichartz,  they also have noticed such arbitrarily sparse behavior of spectra of $\mu_4$ around 2000}.  Our main Theorem \ref{main-thm} now further generalizes the behavior of arbitrarily sparseness of spectra in Beurling dimension to all singular self-affine measures generated by Hadamard triples.

\medskip

The arbitarily sparseness behavior was in stark contrast with the classical cases. The classical result of  Landau \cite{Lan1967} showed that if $\Lambda$ is a spectrum for $L^2(\Omega)$ and $\Omega\subset \R^d$ (or more generally $E(\Lambda)$ is a Fourier frame for $L^2(\Omega)$),  then the $d$-Beurling density of $\Lambda$ must be  at least the Lebesgue measure of $\Omega$ and thus its Beurling dimension must be $d$.  Therefore arbitrarily sparse spectrum does not exist for Lebesgue measure. For a simple proof of Landau's theorem, one can consult  \cite[Chapter 5]{OU}. Landau's theorem is now fundamental in modern sampling theory (see e.g. \cite{Chr,OU} for details). 

\medskip

The arbitrarily sparseness makes us ask naturally if there is any lower bound for the Beurling dimension of the spectra of a spectral measure. The following proposition provides a simple but useful answer. 

\begin{proposition}\label{prop1.4}
Let $\mu$ be a finite  Borel and singluar measure on $\R^d$ such that its Fourier transform satisfies that for $|\xi|$ large enough,
\begin{equation}\label{decay condition}
|\widehat{\mu}(\xi)|^2 \le C|\xi|^{-\gamma}
    \end{equation}
and let $E(\Lambda)$ be a set of exponentials such that there exists $A>0$, 
\begin{equation}\label{eqframelower}
A \|f\|^2 \le \sum_{\lambda\in\Lambda} \left|\int f(x)e^{-2\pi i \lambda\cdot x}d\mu(x)\right|^2 ,\quad  \forall f\in L^2(\mu).
\end{equation}
Then $\gamma\le \dim^{+}(\Lambda)$.
\end{proposition}

It is known that supremum of $\gamma$ such that (\ref{decay condition}) holds is called the {\it Fourier dimension} of $\mu$. This proposition implies that the Beurling dimension is at least the Fourier dimension of $\mu$. For self-affine measure of our consideration, it is easy to prove that all such measures have Fourier dimension zero.

\medskip

Indeed, a stronger result was  proved in \cite{ILLW} in which if (\ref{decay condition}) holds for a spectral measure, then its spectrum must satisfy 
\begin{equation}\label{eqinfinite}
\sum_{\lambda\in \Lambda\setminus\{0\}} |\lambda|^{-\gamma} = \infty.
\end{equation}
One can show that (\ref{eqinfinite}) implies that $\gamma\le \mbox{dim}^{+}(\Lambda)$. We here will provide another independent proof of Proposition \ref{prop1.4}.  

\medskip

Using (\ref{eqinfinite}), one can show that the surface measure of any convex body with everywhere positive Gaussian curvature does not admit any Fourier frame, and  is therefore not spectral. In view of this result, an interesting problem that arises but not yet appeared to have a simple answer is that:

\medskip

{\bf Question:} Does there exist a singular spectral measure whose Fourier dimension is positive? 

\medskip

For some more results about Beurling dimension, Fourier decay and spectral measures, one can also refer to \cite{HKTW,Lev,shi}. In \cite{shi}, the relationship of different dimensions were studied and the above question was also mentioned. 

\medskip

It is also worth mentioning that a widely open problem is to determine if $\mu_3$ admits a Fourier frame or Riesz basis. Beurling dimension has been an indicator to see if such frame is possible to exist \cite{DHW2011}.   It was recently found that it is possible to construct an exponential Riesz sequence (note that a complete Riesz sequence will be a Riesz basis) with maximal Beurling dimension $\log_23$ \cite{DEL}.



\subsection{Sketch of the proof.}
We now sketch the proof of Theorem \ref{main-thm}. First, it is known that in the Hadamard triple, $B$ must be a distinct respresentative in the group $\Z^d/R(\Z^d)$. Therefore, $\#B\le |\det(R)|$.
When $\#B=|\det R|$, then $B$ is a distinct respresentative in the group $\Z^d/R(\Z^d)$.  $\mu$ is just the Lebesgue measure supported on the fundamental domain $T(R, B)$. Hence the spectra of spectral measure $\mu$ has Beurling dimension $d$. Therefore,  $\#B<|\det R|$ is necessary in the assumption. In this case, $\mu$ is singular to the Lebesgue measure. 

\medskip

Throughout the paper, we will assume, without loss of generality, $0\in B\cap L$. Otherwise, we do a translation of the measure. Similar the strategy of the proof in \cite{DHL2019},  for any  singular to the Lebesgue measure $\mu$ on $\R^d$, its {\it periodic zero set} is defined as follows:
$$
{\mathcal Z}(\mu) = \{\xi\in\R^d: \widehat{\mu}(\xi+k) = 0, \ \forall\ k\in\Z^d\}.
$$
Our proof of Theorem \ref{main-thm} is divided into two cases ${\mathcal Z}(\mu)=\emptyset$ or ${\mathcal Z}(\mu)\neq\emptyset$. 

\begin{definition}
We say  that a countable set $\Lambda = \{\lambda_n\}_{n=0}^{\infty}\subset\R^d$ is called  {\it $b$-lacunary}, if $\lambda_0=0$, $|\lambda_1|\ge b$ and for all $n\ge 1$, 
$$
|\lambda_{n+1}|\ge b  |\lambda_n|.
$$
\end{definition}
 As we will see, lacunary sequences must have Beurling dimension (See Proposition \ref{no-dim-of-lacu}). Our theorem in this case is as follows:

\medskip

\begin{theorem}\label{main-thm2}
If $(R, B, L)$ forms a Hadamard triple with $\#B<|\det R|$ and ${\mathcal Z}(\mu(R, B))=\emptyset$, then for all $b>1$, the spectral measure $\mu(R, B)$ admits a b-lacunary spectrum $\Lambda$.   
\end{theorem}

\medskip

The case ${\mathcal Z}(\mu(R, B))\neq\emptyset $ is more complicated. Our strategy is to reduce the self-affine pair $(R, B)$ to a pair $(\widetilde{R}, \widetilde{B})$ which has quasi product-form structure. The  self-affine measure $\mu(\widetilde{R}, \widetilde{B})$ projects on $\R^r$ is a self-affine measure $\mu(\widetilde{R_1}, \widetilde{B_1})$ satisfies ${\mathcal Z}(\mu(\widetilde{R_1}, \widetilde{B_1}))=\emptyset$. Then we can construct a spectrum of $\mu(R, B)$ has zero Beurling dimension.

\medskip

We organize our paper as follows: In section 2, we present some preliminaries. We will review the property of Beurling dimension of a discrete set and the condition ${\mathcal Z}(\mu)=\emptyset$. In section 3 , we will prove Theorem \ref{main-thm2}. In section 4, we will conjugate with some matrix so that $(R, B)$ are of the quasi-product form and complete the proof of  Theorem \ref{main-thm}. We will finally prove Proposition \ref{prop1.4} in the last section. 

\medskip

\section{Preliminaries}
In this section, we will set up some basic propositions for the rest of our paper. These results that serve as the basis for our proofs.

\bigskip

\subsection{Beurling dimension.}  
We will establish some basic properties of Beurling dimension in this subsection. 
\begin{proposition}\label{conjugate}
Let $R$ be an invertible matrix in $M_d(\R)$ and $\Lambda\subset\R^d$ is a discrete set. Then $\dim^+(\Lambda)=\dim^+(R\Lambda).$
\end{proposition}
\begin{proof}
For the invertible matrix $R$, there exist constants $c_1> c_2>0$ such that
$$Q_{c_2}^d(0)\subset R^{-1}Q_1^d(0)\subset Q_{c_1}^d(0).$$
As $h Q_1^d(0) = Q_h^d(0)$, it implies that for all $h>0$, 
$$
Q_{c_2h}^d(0)\subset R^{-1}Q_h^d(0)\subset Q_{c_1h}^d(0).
$$
Note that $\#(\Lambda\cap Q_{c_2h}^d(y)) = \#((\Lambda-y)\cap Q_{c_2h}^d(0)).$
\begin{eqnarray*}
\#(\Lambda\cap Q_{c_2h}^d(R^{-1}x)) &= &\#((\Lambda-R^{-1}x)\cap Q_{c_2h}^d(0))\\
&\le&\#(\Lambda-R^{-1}x)\cap R^{-1}Q_h^d(0))\\
&\le&\#(\Lambda\cap Q_{c_1h}^d(R^{-1}x)).
\end{eqnarray*}
As $\#(R\Lambda \cap Q) = \#(\Lambda\cap R^{-1}Q)$, we have 
$$\#(\Lambda-R^{-1}x)\cap R^{-1}Q_h^d(0))=\#(R\Lambda\cap Q_h^d(x)).$$
We have thus obtained
$$
\#(\Lambda\cap Q_{c_2h}^d(R^{-1}x))\le \#(R\Lambda\cap Q_h^d(x))\le \#(\Lambda\cap Q_{c_1h}^d(R^{-1}x)).
$$
Dividing by $h^{r}$ and taking supremum and limsup, we have
$$c_2^rD^+_r(\Lambda)\le D^+_r(R\Lambda)\le c_1^rD^+_r(\Lambda).$$
Therefore, $\dim^+(\Lambda)=\dim^+(R\Lambda)$ follows.
\end{proof}

\medskip

As one can imagine that a b-lacunary set must be very sparse in ${\mathbb R}^d$, the following proposition gives us the affirmative answer. We will use this frequently in the rest of the paper. 

\medskip

\begin{proposition}\label{no-dim-of-lacu}
Let $b>1$ and $\Lambda$ is a $b-$lacunary set. Then $\dim^+(\Lambda)=0$.
\end{proposition}
\begin{proof} Denote $B(x, r)$ be the open ball centered at $x\in{\mathbb R}^d$ with radius $r>0$. Let ${\mathcal A}_1=B(0, b)$ and ${\mathcal A}_n=B(0, b^n)\setminus B(0, b^{n-1})$ be the annuli regions centered at the origin for any $n\geq2$. We claim that $(\Lambda\setminus\{0\})\cap {\mathcal A}_n$ has at most one element.  In fact, if $\lambda, \lambda'\in (\Lambda\setminus\{0\})\cap {\mathcal A}_n$, then $b^{n-1}\le |\lambda|, |\lambda'|<b^{n}$. It follows that 
$$\frac{|\lambda'|}{|\lambda|}<b,\quad \frac{|\lambda|}{|\lambda'|}<b.$$
This contradicts to the definition of $b$-lacunary set.

\medskip

{\bf Claim:} If $\Lambda$ is $b$-lacunary and given any cubes $Q_h^d(x)$, 
$$
\#(\Lambda\cap Q_h^d(x)) \le \log_b(4\sqrt{d} h)+2.
$$

\medskip
 
\proof For any $h>1$  and $x\in \R^d$, we can find integers  $n_{x, h}\geq0$ and $k_{x, h}\geq1$ such that 
$$Q^d_h(x)\subset \bigcup_{i=1}^{k_{x,h}}{\mathcal A}_{n_{x,h}+i}$$
and $Q^d_h(x) \ \cap \  {\mathcal A}_{n_{x,h}+i}\neq\emptyset$ for all $i = 1,...,k_{x,h}$. Moreover,  $\text{diam}(Q^d_h(x))\geq \text{dist}({\mathcal A}_{n_{x,h}+1}, {\mathcal A}_{n_{x,h}+k_{x,h}})$. When $k_{x,h}>2$, that is  
$$2\sqrt{d}h\geq b^{n_{x,h}+k_{x,h}-1}-b^{n_{x,h}+1}\geq b^{k_{x,h}-2}-1.$$ So
\begin{equation}\label{k_x,h}
k_{x,h}\le \log_b (2\sqrt{d}h+1)+2\le \log_b(4\sqrt{d}h)+2.
\end{equation}
It is clear that when $k_{x,h}\le2$, the above inequality \eqref{k_x,h} also holds.  Combining with the claim in the first paragraph, we have that 
\begin{equation*}
   \#(\Lambda\cap Q_h^d(x)) \le  \#\left(\Lambda\cap \bigcup_{i=1}^{k_{x,h}}{\mathcal A}_{n_{x,h}+i}\right)\le k_{x,h}\le \log_b(4\sqrt{d}h)+2.
\end{equation*}
This justifies the claim. 

\medskip

Hence for any $r>0$, we have
\begin{eqnarray}\label{card-b-lacu}
\limsup_{h\to\infty}\sup_{x\in{\mathbb R}^d}\frac{\#\left(\Lambda\cap Q^d_h(x)\right)}{h^r}
&\le&\lim_{h\to\infty}\frac{\log_b(4\sqrt{d}h)+2}{h^r}=0.
\end{eqnarray}
 That is to say $D^+_r(\Lambda)=0$. From the definition of Beurling dimension, we have $\dim^+(\Lambda)=0$.
\end{proof}

\medskip

\subsection{Periodic Zero set ${\mathcal Z}(\mu)=\emptyset$}

Throughout the paper, the Fourier transform of a Borel probability measure $\mu$ on ${\mathbb R}^d$ is defined to be
$$
\widehat{\mu}(\xi) = \int e^{-2\pi i \xi\cdot x} d\mu(x).
$$
In this subsection, we will  be
devoted to understanding the condition 
$${\mathcal Z}(\mu)=\{\xi\in\R^d: \widehat{\mu}(\xi+k)=0,\ \forall\ k\in\Z^d \}=\emptyset$$
when $\mu$ is a singular measure. The following is an important observation and we will strength the conclusion further when $\mu$ is a self-affine measures in the next section. 

\medskip

\begin{proposition}\label{key-lemma}
 Suppose the periodic zero set ${\mathcal Z}(\mu)$ is empty and $\mu$ is singular. Then for all $\xi\in\R^d$, the set
$$
{\mathcal K}_{\xi} = \{k\in{\mathbb Z}^d:\widehat{\mu}(\xi+k) \ne 0 \}
$$
is an infinite set.
\end{proposition}

\begin{proof}
Note that ${\mathcal K}_{\xi+k_0}={\mathcal K}_{\xi}+k_0$ for any $\xi\in\R^d$ and $k_0\in\Z^d$. So we just need to consider the set ${\mathcal K}_{\xi}$ for $\xi\in[0,1)^d$. Let $\mu_{\xi}$ be the complex measure  $e^{-2\pi i \xi\cdot x} d\mu(x)$. Consider the complex measure on ${\mathbb T}^{d}$, which we identify as $[0,1)^d$,
$$
\nu_{{\mathbb T}^d,\xi} (E)=  \sum_{n\in{\mathbb Z}^d}\mu_{\xi}(E+n)
$$
for all Borel set $E\in{\mathbb T}^d$. Then $\nu_{{\mathbb T}^d,\xi}$ is a measure on ${\mathbb T}^d$ and its Fourier coefficients equal
$$
\widehat{\nu_{{\mathbb T}^d,\xi}}(n) = \widehat{\mu_{\xi}}(n) = \widehat{\mu}(\xi+n)
$$
(For details, see \cite{K2004}). Since ${\mathcal Z}(\mu) $ is empty,  ${\mathcal K}_{\xi}$ is not an empty set and therefore, $\nu_{{\mathbb T}^d,\xi}$ is not a zero measure on ${\mathbb T}^d$. We establish the following claim:

\medskip

{\bf Claim:} $\nu_{{\mathbb T}^d,\xi}$ is singular to the Lebesgue measure on ${\mathbb T}^d$

\medskip

\noindent{\it Proof of claim:} We first note that $\nu_{{\mathbb T}^d,\xi}$ is absolutely continuous with respect to $\nu_{{\mathbb T}^d,0}$. Indeed, if
$\nu_{{\mathbb T}^d,0}(E)=0$, then $\mu(E+n) =0$ for all $n\in {\mathbb Z}^d$. Hence, $\mu_{\xi}(E+n) = 0$ and $\nu_{{\mathbb T}^d,\xi} (E) = 0$ follows. Therefore, we just need to show
that $\nu_{{\mathbb T}^d,0}$ is singular with respect to the Lebesgue measure.

\medskip

Let $m_{{\mathbb R}^d}$ and $m_{{\mathbb T}^d}$ be the Lebesgue measure on $\R^d$ and ${\mathbb T}^d$ respectively. Since $\mu$ is singular to $m_{\R^d}$, we can find a set $A$ such that $m_{\R^d}(A) = 0$ and $\mu$ is supported on $A$. For each $n\in{\mathbb Z}^d$,  let
$$A_n = \left(A\cap ([0,1)^d+n)\right)- n,$$
so that
$$
A = \bigcup_{n\in{\mathbb Z}^d} (A_n+n).
$$
Let $E = \bigcup_{n\in{\mathbb Z}^d} A_n$ and let $ {\mathbb T}^d\setminus E$ be its complement. Then
$$
\nu_{{\mathbb T}^d,0} ({\mathbb T}^d\setminus E) = \sum_{n\in{\mathbb Z}^d} \mu \left(\bigcap_{n\in{\mathbb Z}^d}\left( {\mathbb T}^d\setminus A_n\right) +n\right) \le \sum_{n\in{\mathbb Z}^d} \mu \left(\left({\mathbb T}^d\setminus A_n\right) +n\right)=0
$$
since $\mu$ is supported on $A_n+n$ on $[0,1)^d+n$. Hence, we know $\nu_{{\mathbb T}^d,0}$ is supported on $E$. But we know that
$$
m_{{\mathbb T}^d}(E) \le \sum_{n\in{\mathbb Z}^d} m_{{\mathbb T}^d} (A_n) = \sum_{n\in{\mathbb Z}^d} m (A_n+n) =0.
$$
This shows that $\nu_{{\mathbb T}^d,0}$ is a singular measure with respect to $m_{{\mathbb T}^d}.$ This justifies the claim.

\medskip

We finally argue by contradiction. Suppose that $\#{\mathcal K}_{\xi}<\infty$. Then we can find some $N_0$ such that  the measure $\widehat{\nu_{{\mathbb T}^d,\xi}}(n)  = 0$ for all $|n|>N_0$.  Then we know that the Fourier coefficients $\{\widehat{\nu_{{\mathbb T}^d,\xi}}(n): n\in\Z^d\}$ is square-summable. By the unitary isomorphism of $L^2({\mathbb T}^d)$ and $\ell^2({\mathbb Z}^d)$, we can find $f\in L^2({\mathbb T}^d)$, in fact a trigonometric polynomial, such that
$$
\widehat{\nu_{{\mathbb T}^d,\xi}}(n)= \widehat{f}(n).
$$
This means that $ \nu_{{\mathbb T}^d,\xi} = f(x)dx$, which is a contradiction since $\nu_{{\mathbb T}^d,\xi}$ is singular to the Lebesgue measure.
\end{proof}

\medskip

We remark that the above proposition is clearly false if $\mu$ is not singular. For example, if $\mu$ is the Lebesgue measure on $[0,1]$. Then ${\mathcal K}_0 = \{0\}$ only since the Fourier transform on the characteristic function of the unit interval is equal to zero on all non-zero integers. 

\medskip

\begin{proposition}\label{equi-positive}
Suppose that the periodic zero set ${\mathcal Z}(\mu)$ is empty. Then there exists $\epsilon>0$ and $\delta>0$ such that for all $\xi\in [0, 1]^d$, there exists $k_{\xi}\in\Z^d$ such that
$$|\widehat{\mu}(\xi+y+k_{\xi})|\geq\epsilon$$
whenever $|y|<\delta$. Specially, for $\xi=0$, we can take $k_0=0$.
\end{proposition}

\medskip

\begin{proof}
As ${\mathcal Z}(\mu)$ is empty, for any $\xi\in [0,1]^d$, we can find $k_{\xi}\in\Z^d$ and $\epsilon_{\xi}>0$ such that
$$
|\widehat{\mu}(\xi+k_{\xi})|\ge \epsilon_{\xi}>0.
$$
By the continuity of $\widehat{\mu}$. we can find  $\delta_{\xi}>0$ such that for all $|y|\le \delta_{\xi}$, we have
$$
|\widehat{\mu}({\xi}+y+k_{\xi})|\ge \frac{\epsilon_{\xi}}{2}.
$$
As $[0, 1]^d \subset \bigcup_{{\xi}\in [0, 1]^d} B({\xi},\delta_{\xi}/2)$, by the compactness of $[0, 1]^d$, we can find ${\xi}_1,...,{\xi}_N\in [0, 1]^d$ such that $[0, 1]^d \subset B({\xi}_1,\delta_{{\xi}_1}/2)\cup...\cup B({\xi}_N,\delta_{{\xi}_N}/2)$. We now take
$$
\delta = \min\left\{\frac{\delta_{{\xi}_j}}{2}: j=1,...,N\right\},  \
\epsilon = \min\left\{\frac{\epsilon_{{\xi}_j}}{2}: j=1,...,N\right\}.
$$
Now, $\delta$ and $\epsilon$ are positive and independent of ${\xi}\in [0, 1]^d$. We claim that the stated property holds. Indeed, for any ${\xi}\in [0, 1]^d$, ${\xi}\in B({\xi}_j,\delta_{{\xi}_j}/2)$ for some $j=1,...,N$. 
Hence,
$$
|\widehat{\mu}({\xi}+k_{{\xi}_j})| = |\widehat{\mu}({\xi}_j+({\xi}-{\xi}_j)+k_{{\xi}_j})|\ge \frac{\epsilon_{{\xi}_j}}{2}\ge \epsilon.
$$
Therefore, we just redefine $k_{{\xi}} = k_{\xi_j}$ to obtain our desired conclusion.
\end{proof}
\medskip

\section{Proof of Theorem \ref{main-thm2}}
 
 In this section, we  first outline how one can  construct a Fourier basis for the self-affine measure $\mu(R, B)$ and then prove Theorem \ref{main-thm2} that we can find $b$-lacunary spectra if the periodic zero set is empty. 
 
 \medskip
 
 Recall that for the self-affine measure $\mu = \mu(R,B)$, by iterating the invariance identity, its Fourier transform can be expressed as an infinite product 
 $$
 \widehat{\mu}(\xi) = \prod_{n=1}^{\infty} \widehat{\delta_{B}}((R^{t})^{-n}\xi) =  \prod_{n=1}^{\infty} \widehat{\delta_{R^{-n}B}}(\xi).
 $$
Here $\delta_A$ denotes the equal-weighted Dirac mass supported on the finite set $A$. From this infinite product, we obtain another expression of the self-affine measure through an infinite convolution of atomic measures
 $$
 \mu(R,B) = \delta_{R^{-1}B}\ast\delta_{R^{-2}B}\ast.... = w-\lim_{n\to\infty} (\delta_{R^{-1}B}\ast\delta_{R^{-2}B}\ast...\ast \delta_{R^{-n}B})
 $$
where w-$\lim$ is the weak limit of the probability measures.

\bigskip

 Given a subsequence of positive integers $\{n_k\}$, we define  ${\bf B}_{n_k}=B+R B+\cdots+ R^{n_{k}-1}B$. Letting $m_k=n_1+\cdots+n_k$.
Then the self-affine measure $\mu = \mu(R,B)$ can be factorized along this subsequence as
\begin{equation}\label{eq_factorize}
    \mu = \delta_{R^{-m_1}{\bf B}_{n_1}}\ast\delta_{R^{-m_2}{\bf B}_{n_2}}\ast\cdots.
\end{equation}
Define also
$$
\mu_{>{k}} = \delta_{R^{-m_{k+1}}{\bf B}_{n_{k+1}}}\ast\delta_{R^{-m_{k+2}}{\bf B}_{n_{k+2}}}\ast\cdots.
$$
The following lemma is  known, whose proof can be found in (\cite{LW2017},  Proposition 3.1).

\bigskip

\begin{lemma}\label{Hadamard-triple}
Suppose $(R, B, L)$ forms a Hadamard triple, for any $n\geq1$, \\
{\rm (i)} then $(R^n, {\bf B}_n, {\bf L}_n)$ is also a Hadamard triple;\\
{\rm (ii)} if $\widetilde{{\bf L}_n}\equiv {\bf L}_n \pmod{(R^t)^n}$, then $(R^n, {\bf B}_n, \widetilde{{\bf L}_n})$ is also a Hadamard triple.
\end{lemma}

\bigskip

We note that $(R,B,L)$ forms a Hadamard triple if and only if $\{e^{2\pi i \ell\cdot x}:\ell\in L\}$ will form an orthonormal basis for  $L^2(\delta_{R^{-1}B})$. Hence, since we know $\{(R^{n_k},{\bf B}_{n_k}, {\bf L}_{n_k})\}$ form Hadamard triples, we define
\begin{equation}\label{eqLambda_k}
\Lambda_k ={\bf L}_{n_1}+(R^t)^{m_1}{\bf L}_{n_2}+\cdots+(R^t)^{m_{k-1}}{\bf L}_{n_k}, \ \mbox{and} \ \Lambda = \bigcup_{k=1}^{\infty}\Lambda_k,
    \end{equation}
and $\Lambda$ forms a mutually orthogonal set for $L^2(\mu(R,B))$. 
The following is the main theorem  giving a sufficient condition for an orthgonal set to be complete \cite{DHL2019,LW2017}.

\medskip

\begin{theorem}[\cite{DHL2019, LW2017}]\label{theorem_LW}
Let $(R,B,L)$ be a Hadamard triple. Let $\Lambda_k$ and $\Lambda$ be defined as in (\ref{eqLambda_k}).
Suppose that
\begin{equation}\label{tailterm}
\delta(\Lambda): = \inf_{k\ge 1}\inf_{\lambda_k\in\Lambda_k} |\widehat{\mu}_{>k}(\lambda_k)|^2 >0.
\end{equation}
Then the self-affine measure $\mu$ is a spectral measure with a spectrum $\Lambda$ in $\Z^d$.
\end{theorem}

\medskip

Condition (\ref{tailterm}) is a sufficient condition guaranteeing the mutually orthogonal sets to be complete. This condition was first proposed by Strichartz \cite{S1998,S2000}. In general, it cannot be removed. On the other hand, this condition is also not necessary \cite{DHL2013}. The following theorem provides a strengthened result of Proposition \ref{key-lemma} in the case of self-affine measures.  

\medskip

\begin{theorem}\label{th1.3}
Suppose a self-affine measure $\mu:=\mu(R, B)$ satisfies ${\mathcal Z}(\mu)=\emptyset$. Then there is a $\epsilon_0>0$ such that for any $\xi\in [0, 1]^d$,
$$
{\mathcal K}_{\xi, \epsilon_0} = \{k\in{\mathbb Z}^d: |\widehat{\mu}(\xi+k)| \geq \epsilon_0 \}
$$
is an infinite set.
\end{theorem}
\begin{proof}
Take $\epsilon>0$ and $\delta>0$ be the constants defined as in Proposition \ref{equi-positive}. For $\xi=0$,  from Proposition \ref{key-lemma}, ${\mathcal K_0}$ is an infinite set, we can choose a $t_0\in{\mathcal K_0}\setminus\{0\}$. Then
there exists $\epsilon'>0$ and $\delta'>0$ such that
\begin{equation}\label{eq2.2}
    |\widehat{\mu}(y+t_0)|\geq\epsilon', \quad \forall\ |y|<\delta'.
\end{equation}
Let $\epsilon_0=\epsilon\cdot\epsilon'$ and $\delta_0=\min\{\delta, \delta'\}$ which are positive constants independent of $\xi\in [0,1]^d$.

\bigskip

From Proposition \ref{equi-positive}, for any $\xi\in [0,1]^d$, there is a $k_{\xi}\in{\mathbb Z}^d$ such that
\begin{equation}\label{eq2.1}
    |\widehat{\mu}(\xi+k_{\xi})| \geq \epsilon.
\end{equation}
Fix $\xi\in [0,1]^d$ and $k_{\xi}$, we can find an integer $n_0\geq1$ such that $|(R^t)^{-n_0}(\xi+k_{\xi})|< \delta_0$. When $n\geq n_0$, since $|(R^t)^{-n}(\xi+k_{\xi})|< \delta_0$, the inequality \eqref{eq2.2} implies that
\begin{equation*}
    |\widehat{\mu}((R^t)^{-n}(\xi+k_{\xi})+t_0)|\geq \epsilon'.
\end{equation*}
It together with inequality \eqref{eq2.1}, we have
\begin{eqnarray*}
|\widehat{\mu}(\xi+k_{\xi}+(R^t)^{n}t_0)|&=&\prod_{k=1}^{\infty}|M_{B}((R^t)^{-k}(\xi+k_{\xi}+(R^t)^nt_0))|\\
&=&\prod_{k=1}^n|M_{B}((R^t)^{-k}(\xi+k_{\xi}))|\cdot\prod_{k=1}^{\infty}|M_B((R^t)^{-k}((R^t)^{-n}(\xi+k_{\xi})+t_0))|\\
&\geq&|\widehat{\mu}(\xi+k_{\xi})|\cdot|\widehat{\mu}((R^t)^{-n}(\xi+k_{\xi})+t_0)|\\
&\geq&\epsilon\cdot\epsilon'=\epsilon_0.
\end{eqnarray*}
The above inequality implies that
$$\{k_{\xi}+(R^t)^n t_0\}_{n=n_0}^{\infty}\subset {\mathcal K}_{\xi, \epsilon_0}.$$
As $(R^t)^k$ is an expanding matrix, $1$ is not an eigenvalue of $(R^t)^k$. So $(R^t)^nt_0\neq (R^t)^mt_0$ when $n\neq m$. Hence ${\mathcal K}_{\xi, \epsilon_0}$ is an infinite set.
\end{proof}

\medskip

Next we prove  Theorem \ref{main-thm2}.

\medskip

\begin{proof} ({\bf Proof of Theorem \ref{main-thm2}}). Let $\epsilon_0>0$ be a constant defined as in Theorem \ref{th1.3}. Then the uniform continuity of $\widehat{\mu}$ implies that there exists a $\delta_0>0$ such that for all $\xi\in [0, 1]^d$ and $|y|<\delta_0$, we have
\begin{equation}\label{eq3.1}
    |\widehat{\mu}(\xi+y+k)|\geq\epsilon_0/2
\end{equation}
holds for $k\in{\mathcal K}_{\xi, \epsilon_0}$.

\bigskip

We now  construct inductively $\Lambda_k$ as in (\ref{eqLambda_k}) so that $\Lambda_k$ is $b-$lacunary and $\delta(\Lambda)\geq\epsilon_0/2>0$. From Lemma \ref{Hadamard-triple} (ii), without loss of generality, we assume $L=\{l_0=0, l_1, \cdots, l_{q-1}\}\subset R^t[0, 1)^d\cap\Z^d$. Then
$$(R^t)^{-n}{\bf L}_n\subset[0, 1)^d.$$
Denote $n_1=1$. Let
$$\lambda_1=l_1+(R^t) k_{(R^t)^{-1}l_1} \quad k_{(R^t)^{-1}l_1}\in {\mathcal K}_{(R^t)^{-1}l_1, \epsilon_0}.$$
For $2\le i\le q-1$, since ${\mathcal K}_{(R^t)^{-1}l_i, \epsilon_0}$ is infinite, we can take a $k_{(R^t)^{-1}l_i}\in{\mathcal K}_{(R^t)^{-1}l_i, \epsilon_0} $ such that
$$\lambda_i=l_i+R^t k_{(R^t)^{-1}l_i} \text{ and } {|\lambda_i|}\geq b{|\lambda_{i-1}|}.$$
Then $\Lambda_1=\{\lambda_0=0, \lambda_1, \lambda_2, \cdots, \lambda_{q-1}\}$ is $b$-lacunary and from \eqref{eq3.1}
$$|\widehat{\mu}_{>1}(\lambda_i)|=|\widehat{\mu}((R^t)^{-1}l_i+k_{(R^t)^{-1}l_i})|\geq \epsilon_0/2, \quad\ 1\le i\le q-1 .$$
Suppose that $\Lambda_{k-1}$ has been constructed which is a $b-$lacunary set
and
$$\inf_{\lambda_{k-1}\in\Lambda_{k-1}}|\widehat{\mu}_{>(k-1)}(\lambda_{k-1})|\geq\epsilon_0/2.$$
We can take a large enough $n_k$ in the subsequence with the following happen:
$$\sup_{\lambda_{k-1}\in\Lambda_{k-1}}\|(R^t)^{-m_k}\lambda_{k-1}\|<\delta_0.$$
(Recall that $m_k = n_1+...+n_k$) We now define
$$\Lambda_k=\Lambda_{k-1}+\{(R^t)^{m_{k-1}}l_{k}+(R^t)^{m_k}k_{x_{l_k}}: l_{k}\in {\bf L}_{n_k}\}$$
where $x_{l_k}=(A^t)^{-n_k}l_k\in[0, 1)^d$ and $k_0=0$. Then $\Lambda_{k-1}\subset\Lambda_k$. As  ${\mathcal K}_{x_l, \epsilon_0}$ is an infinite set, by choosing $k_{x_l}\in{\mathcal K}_{x_l, \epsilon_0}$ as large as we wanted, we can ensure that 
$\Lambda_{k}$ is $b$-lacunary. Now writing
$$\lambda_k=\lambda_{k-1}+(R^t)^{m_{k-1}}l_k+(R^t)^{m_k}k_{x_{l_k}},$$
for some $\lambda_{k-1}\in\Lambda_{k-1}$, we have
\begin{eqnarray*}
|\widehat{\mu_{>k}}(\lambda_k)|&=&|\widehat{\mu}((R^t)^{-m_k}\lambda_k)|\\
&=&|\widehat{\mu}((R^t)^{-m_k}\lambda_{k-1}+x_{l_k}+k_{x_{l_k}})| \ \ (\text{using} \  \eqref{eq3.1})\\
&\geq& \epsilon_0/2>0.
\end{eqnarray*}
Hence, $\delta(\Lambda)>0$ is now satisfied. The $b$-lacunary set $\Lambda=\bigcup_{k=1}^{\infty}\Lambda_k$ is a spectrum of $\mu$ according to Theorem \ref{theorem_LW}.
\end{proof}






\bigskip

As a corollary, we settle the case for the self-similar measure on $\R^1$.

\begin{corollary}\label{one-dimension}
 Let $R>1$ be an integer and $B\subset\Z$ be a digit set with $\# B< R$ and gcd$(B) = 1$. Suppose that $(R, B, L)$ forms a Hadamard triple. Then ${\mathcal Z}(\mu(R, B))=\emptyset$ and $\mu(R, B)$ admits a spectrum $\Lambda$ with $\dim^+(\Lambda)=0$.
\end{corollary}
\begin{proof} It has been proved in \cite[Section 5]{DHL2019}  that if gcd$(B)$=1, then the periodic zero set of  $\mu(R, B)$ is empty and therefore it has a $b-$lacunary  spectrum in $\Z$ by Theorem \ref{main-thm2}, which has Beurling dimension zero by Proposition \ref{no-dim-of-lacu}. 
\end{proof}

\bigskip

\section{Proof of Theorem \ref{main-thm}.}

\bigskip

We first discuss some preliminary reduction that we can perform in order to prove our main theorem.

\bigskip

\begin{definition}
Let $R_1, R_2$ be $d\times d$ integer matrices, and the finite sets $B_1, B_2, L_1, L_2$ be in ${\mathbb Z}^d$. We say that two triples $(R_1, B_1, L_1)$ and $(R_2, B_2, L_2)$ are {\it conjugate} (through the matrix $M$) if there exists an integer unimodular  matrix $M$  such that $R_2=MR_1M^{-1}, B_2=MB_1$ and $L_2=(M^t)^{-1}L_1$. (Here, unimodular matrix means its determinant is 1).
\end{definition}

\bigskip

The following proposition is obtained from some simple computations.

\bigskip

\begin{proposition}\label{mu_1-mu_2}
Suppose that $(R_1, B_1, L_1)$ and $(R_2, B_2, L_2)$ are two conjugate triples, through the matrix $M$. Then

{\rm(i)} If $(R_1, B_1, L_1)$ is a Hadamard triple then so is $(R_2, B_2, L_2)$.

{\rm(ii)} The measure $\mu(R_1, B_1)$ is spectral with spectrum $\Lambda$ if and only if $\mu(R_2, B_2)$ is spectral with spectrum $(M^t)^{-1}\Lambda$.

{\rm(iii)} Spectral measures $\mu(R_1, B_1)$ and $\mu(R_2, B_2)$ have spectrum $\Lambda$ with $\dim^+(\Lambda)=0$ simultaneously.
\end{proposition}
\begin{proof}
The proof of {\rm(i), \rm(ii)} can be found in e.g. (\cite{DJ2007}, Proposition 3.4). The {\rm(iii)} follows from the fact that $\dim^+(\Lambda)=\dim^+((M^t)^{-1}\Lambda)$ which is proved in Proposition \ref{conjugate}

\end{proof}

\bigskip

 We use $E\times F$ to denote the Cartesian product of $E$ and $F$ so that $E\times F=\{(e, f)^t: e\in E, f\in F\}$. We first introduce the following notations.

\begin{definition}
For a vector $x\in\R^d$, we write it as $x=(x_1, x_2)^t$ with $x_1\in\R^r$ and $x_2\in\R^{d-r}$. We denote by $\pi_1(x)=x_1, \pi_2(x)=x_2$. For a subset $E$ of $\R^d$, and $x_1\in\R^r, x_2\in\R^{d-r}$, we denote by
$$E_2(x_1):=\{y\in\R^{d-r}: (x_1, y)^t\in E\}, E_1(x_2):=\{x\in\R^r: (x, x_2)^t\in E\}.$$
\end{definition}

\bigskip

We define ${\mathbb Z}[R, B]$ to be the smallest $R-$invariant lattice containing all $\sum_{j=0}^{n-1}R^jB$. To prove theorem \ref{main-thm},  there is no loss of generality to assume that ${\mathbb Z}[R, B]={\mathbb Z}^d$ since we can always conjugate the Hadamard triple to produce a Hadamard triple with ${\mathbb Z}[R, B]={\mathbb Z}^d$. If $(R, B)$ and $(\widetilde{R}, \widetilde{B})$ are conjugate through an integer unimodular matrix $M$, then 
$$\Z[\widetilde{R}, \widetilde{B}]=M\Z[R, B]=M\Z^d=\Z^d.$$

\medskip

By studying the dynamical system underlying the self-affine system, the following decomposition was proved in \cite[Section 6 and 7]{DHL2019}.  If $(R, B, L)$ is a Hadamard triple  such that $\Z[R,B] = \Z^d$ and ${\mathcal Z}(\mu)\neq\emptyset$, we can always conjugate with some integer unimodular matrix so that $(R, B)$ are of the following {\it quasi-product form}:
\begin{equation}\label{R-quasi-product}
    R=\left(
  \begin{array}{cc}
   R_1  &  0 \\
    C_0 & G_1 \\
  \end{array}
\right)
\end{equation}

\begin{equation}\label{B-quasi-product}
    B=\{(u_i, d_{j}(u_i))^t: 1\le i\le N_1, 1\le j\le N_2:=|\det G_1|\},
\end{equation}
and $\{d_{j}(u_i): 1\le j\le N_2\}$ is a complete set of representatives $\pmod{G_1\Z^{d-r}}$.
Note that $R^{-k}$ can be written in the following form
$$R^{-k}=\left(
  \begin{array}{cc}
   R_1^{-k}  &  0 \\
    C_k   &G_1^{-k} \\
  \end{array}
\right)$$
for some $C_k\in M_{d-r, r}(\R)$. For the self-affine set $T(R, B)$, we can express it as a set of infinite sums
$$T(R, B)=\left\{\sum_{k=1}^{\infty}R^{-k}b_k: b_k\in B\right\}.$$
Therefor any element $(x_1, x_2)^t\in T(R, B)$ can be written as
$$x_1=\sum_{k=1}^{\infty}R_1^{-k}u_{i_k},\ x_2=\sum_{k=1}^{\infty}C_ku_{i_k}+\sum_{k=1}^{\infty}G_1^{-k}d_{j_k}(u_{i_k}).$$
Hence
$$\pi_1(T(R, B))=T(R_1, \pi_1(B))$$
is a self-affine set where $\pi_1(B)=\{u_i: 1\le i\le N_1\}$. For each
$$x_1=\sum_{k=1}^{\infty}R_1^{-k}u_{i_k}\in T(R_1, \pi_1(B)),$$
we have
$$\left(T(R, B)\right)_2(x_1)=\sum_{k=1}^{\infty}C_ku_{i_k}+\left\{\sum_{k=1}^{\infty}G_1^{-k}d_{j_k}(u_{i_k}): 1\le j_k\le N_2\right\}$$
almost surely tiles $\R^{d-r}$.
Moreover, we can find a $L'\equiv L\pmod{R^t}$ such that  $(R_1, \pi_1(B), L_1'(l_2))$ is a Hadamard triple on $\R^r$ for all $l_2\in\pi_2(L')$.
Let $\mu(R_1, \pi_1(B))$ be the self-affine measure supported on $T(R_1, \pi_1(B))$ and $\mu_2^{(x_1)}$ be  the infinite convolution product 
$$
\delta_{G_1^{-1}B_2(i_1)}\ast\delta_{G_1^{-2}B_2(i_2)}\ast\dots,$$
where  $B_2(i_k):=\{d_{j}(u_{i_k}) : 1\leq j\leq N_2\}$
We call $\mu_2^{(x_1)}$ the {\it Cantor-Moran measure} supported on $\left(T(R, B)\right)_2(x_1) - \sum_{k=1}^{\infty}C_ku_{i_k}$. the measures $\mu(R_1,\pi_1(B))$ and  $\mu_2^{(x_1)}$ are called a {\it quasi-product form decomposition} of $\mu(R, B)$.

\bigskip

\begin{theorem}\label{quasi-prod-form-decomp}
 Let $(R, B, L)$ be a Hadamard triple on $\R^d$ such that $\Z[R,B] = \Z^d$ and ${\mathcal Z}(\mu)\neq\emptyset$. Then we can conjugate with an integer unimodular matrix $M$ so that $\mu(MRM^{-1}, MB)$ has a quasi-product form $\mu_1$ and $\mu_2^{(x_1)}$ on $\R^r\times\R^{d-r}$ with ${\mathcal Z}(\mu_1)=\emptyset$.
\end{theorem}

\smallskip

\begin{proof} We will prove the result by using induction on dimension $d$. When $d=1$, the assumption $\Z[R, B]=\Z$ forces $\gcd(B)=1$. In this case, it has been proved in \cite[Section 5]{DHL2019} that ${\mathcal Z}(\mu)=\emptyset$ .
Hence, if ${\mathcal Z}(\mu)\neq\emptyset$, then $d\geq2$.
\bigskip

For $d=2$, as we have discussed, we can conjugate with an integer unimodular matrix $M\in M_2(\Z)$ so that $(MRM^{-1}, MB)$ are of the quasi-product form on $\R\times\R$ as following:
\begin{equation*}
    MRM^{-1}=\left(
  \begin{array}{cc}
   R_1  &  0 \\
   C_1 & G_1 \\
  \end{array}
\right)
\end{equation*}
\begin{equation*}
    MB=\{(u, d_{j}(u))^t: u\in B_1, 1\le j\le |G_1|\}\subset\Z^{2},
\end{equation*}
and  $\{d_{j}(u): 1\le j\le |G_1|\}\subset\Z$ is a complete set of representatives $\pmod{|G_1|}$ (Here $|G_1|\ge 2$ is a positive integer). Moreover,  $\gcd(B_1)=1$ since $\Z[MRM^{-1}, MB]=\Z^2$.
Then $\mu(MRM^{-1}, MB)$ has a quasi-product form $\mu_1$ and $\mu_2^{(x_1)}$ on $\R\times\R$ where $\mu_1=\mu(R_1, B_1)$ is the self-similar measure supported on $T(R_1, B_1)\subset\R$. Since $\gcd(B_1)=1$ and  $(R_1, \pi_1(B), L_1'(l_2))$ forms a Hadamard triple on $\R$ for some $L'$, $l_2\in\pi_2(L')$, we have ${\mathcal Z}(\mu_1)=\emptyset$. This justifies the result  on dimension two. 


\bigskip

Assume our statement 
is true for any dimensions less than $d$. On dimension $d$, we can conjugate with an integer unimodular matrix $M_1\in M_{d}(\Z)$ so that $(M_1RM_1^{-1}, M_1B)$ has quasi-product form
\begin{equation*}
    M_1RM_1^{-1}=\left(
  \begin{array}{cc}
   R_1  &  0 \\
    C_1 & G_1 \\
  \end{array}
\right)
\end{equation*}
\begin{equation*}
    M_1B=\{(u, d_{j}(u))^t: u\in B_1, 1\le j\le |\det G_1|\}\subset\Z^{d},
\end{equation*}
where $\Z[R_1, B_1]=\Z^r$ and $\{d_{j}(u): 1\le j\le |\det G_1|\}\subset\Z^{d-r}$ is a complete set of representatives $\pmod{G_1\Z^{d-r}}$. If ${\mathcal Z}(\mu(R_1, B_1))=\emptyset$, then $\mu_1=\mu(R_1, B_1)$ and $\mu_2^{(x_1)}$ is the desired quasi-product form on $\R^{r}\times \R^{d-r}$.

\bigskip

If ${\mathcal Z}(\mu(R_1, B_1))\neq\emptyset$, by the assumption on $r<d$, we can conjugate with an unimodular matrix $M_2\in M_r(\Z)$ such that 
\begin{equation*}
    M_2R_1M_2^{-1}=\left(
  \begin{array}{cc}
   R_2  &  0 \\
    C_2 & G_2 \\
  \end{array}
\right)
\end{equation*}
\begin{equation*}
    M_2B_1=\{M_2u:u\in B_1\}=\{(v, w_j(v))^t: v\in B_2, 1\le j\le |\det G_2|\}\subset\Z^{r},
\end{equation*}
where $\{w_j(v): 1\le j\le |G_2|\}\subset\Z^{r-r_1}$ is a complete set of representatives $\pmod{G_2\Z^{r-r_1}}$ and ${\mathcal Z}(\mu(R_2, B_2))=\emptyset.$
Denote $$M=\tiny{\left(
  \begin{array}{cc}
   M_2  &  0 \\
    0     &  \text{I}_{d-r} \\
  \end{array}
\right)}M_1$$
where $\text{I}_{d-r}$ is the identity matrix and rewrite $d_{i}(u)$ as $d_i(v,w_j(v))$ if $M_2u=(v, w_j(v))^t$. Then $(MRM^{-1}, MB)$ has quasi-product form on $\R^{r_1}\times \R^{d-r_1}$ as following
\begin{equation*}
    MRM^{-1}=\left(
  \begin{array}{cc}
   R_2  &  0 \\
    C_2' & G_2' \\
  \end{array}
\right)
\end{equation*}
where $$G_2'=\left(
  \begin{array}{cc}
   G_2  &  0 \\
    C_2'' & G_1 \\
  \end{array}
\right)$$
for some matrix $C_2''\in M_{d-r, r-r_1}(\Z)$ and 
\begin{equation*}
    MB=\Big\{(v,w_{j_2}(v), d_{j_1}(v,w_{j_2}(v))^t: v\in B_2, 1\le j_2\le |\det G_2|, 1\le j_1\le |\det G_1|\Big\},
\end{equation*}
where $$\{(w_{j_2}(v), d_{j_1}(v,w_{j_2}(v))^t: 1\le j_2\le |\det G_2|, 1\le j_1\le |\det G_1|\}\subset\Z^{d-r_1}$$ is a complete set of representatives $\pmod{G_2'\Z^{d-r_1}}$. 
As ${\mathcal Z}(\mu(R_2, B_2))=\emptyset,$ $\mu_1=\mu(R_2, B_2)$ and $\mu_2^{(x_1)}$ is the desired quasi-product form decomposition of the self-affine measure $\mu(R,B)$ on $\R^{r_1}\times\R^{d-r_1}$.

\end{proof}

\bigskip

The spectrality for $\mu_1$ and $\mu_2^{(x_1)}$ in the quasi-product-form decomposition  was proved in \cite{DHL2019}

\begin{theorem}\cite[Proposition 8.4]{DHL2019}
Suppose $(R, B, L)$ is a Hadamard triple and ${\mathcal Z}(\mu)\neq\emptyset$. Let $\mu_1, \mu_2^{(x_1)}$ be a quasi-product form of $\mu$ on $\R^{r}\times\R^{d-r}$, then $\mu_1$ is spectral  and for $\mu_1-$almost every $x_1\in T(R_1, B_1)$, $\mu_2^{(x_1)}$ admits a spectrum $\Gamma$ which is a full-rank lattice in $\R^{d-r}$.
\end{theorem}

\bigskip

The following  lemma is proved in \cite{DJ2007}.

\bigskip

\begin{lemma}[\cite{DJ2007}, Lemma 4.4]\label{product-integral}
If $\Lambda_1$ is a spectrum for the measure $\mu_1$, then for all $x_2\in \R^{d-r}$
$$\sum_{\lambda_1\in\Lambda_1}|\widehat{\mu}(x_1+\lambda_1, x_2)|^2=\int_{T(A_1, \pi_1(B))}|\widehat{\mu}_{2}^{(s)}(x_2)|^2d\mu_{1}(s).$$

\end{lemma}

\bigskip

We recall also the Jorgensen-Pedersen Lemma for checking when a countable set is a
spectrum for a measure.

\bigskip

\begin{lemma}\cite{JP1998}\label{JP-crit}
Let $\mu$ be a  compactly supported probability measure on $\R^d$. Then a countable set $\Lambda$ is a spectrum for $L^2(\mu)$ if and only if
$$\sum_{\lambda\in\Lambda}|\widehat{\mu}(\xi+\lambda)|^2\equiv1, \quad\ \xi\in\R^d.$$
\end{lemma}

\bigskip

We now state a general class of spectrum for the quasi product-form. 
\begin{theorem}\label{product-spectrum}
Suppose $\Gamma$ is a spectrum of $\mu_2^{(x_1)}$ for $\mu_1-$almost every $x_1$ and $\{\Lambda_{\gamma}\}_{\gamma\in\Gamma}$ is a class of spectra of $\mu_1$. Then $\bigcup_{\gamma\in\Gamma}\left(\Lambda_{\gamma}\times\{\gamma\}\right)$ is a spectrum of $\mu$.
\end{theorem}
\begin{proof}
Since $\Lambda_{\gamma}$ is a spectrum of $\mu_1$, by Lemma \ref{product-integral}, we have
$$\sum_{\gamma\in\Gamma}\sum_{\lambda\in\Lambda_{\gamma}}|\widehat{\mu}(x_1+\lambda, x_2+\gamma)|^2=\int_{T(A_1, \pi_1(B))}\sum_{\gamma\in\Gamma}|\widehat{\mu}_2^{(s)}(x_2+\gamma)|^2d\mu_1(s)=\int_{T(A_1, \pi_1(B))}1d\mu_1(s)=1.$$
This means that $\bigcup_{\gamma\in\Gamma}\left(\Lambda_{\gamma}\times\{\gamma\}\right)$ is a spectrum of $\mu$ by Lemma \ref{JP-crit}.
\end{proof}

\bigskip
In \cite{DHL2019}, only the case that all $\Lambda_{\gamma}$ are the same was considered. However, to construct zero Beurling dimension spectra, $\Lambda\times\Gamma$ is not enough since $\{0\}\times\Gamma$ is contained in the spectrum and this will contribute to the Beurling dimension $d-r$ since $\Gamma$ is lattice on $\R^{d-r}$. To overcome this problem, we will need different sparse spectra for each $\gamma$.  Now we can prove our main result.

\bigskip

\begin{proof} ({\bf Proof of Theorem \ref{main-thm}})
If ${\mathcal Z}(\mu)=\emptyset$, then the result follows from Theorem \ref{main-thm2}. Suppose now that ${\mathcal Z}(\mu)\neq\emptyset$. Since conjugation maintains the Beurling dimension of $\Lambda$,  without loss of generality, we assume $\mu$ has quasi-product form $\mu_1=\mu(R_n, B_n), \mu_2^{(x_1)}$ on $\R^r\times\R^{d-r}$ with ${\mathcal Z}(\mu_1)=\emptyset$ from Theorem \ref{quasi-prod-form-decomp}. Moreover, $(R_n, B_n, L_n)$ forms a Hadamard triple. From Theorem \ref{main-thm2},
for any $b>1$, we can construct a $b-$lacunary spectrum of $\mu_1$.

\bigskip

Let the full-rank lattice $\Gamma=A\Z^{d-r}$ is a spectrum of $\mu_2^{(x_1)}$ for $\mu_1-$almost every $x_1\in T(R_1, B_1)$. We choose a sequence of vectors $\{m^0_n\}_{n=0}^{\infty}\subset\Z^{d-r}$ such that $m^0_0=0$ and $|m^0_n|=2^{n-1}$. Let $N_{0, r}=0$ and  
$$
N_{n,r} = \#\{m\in{\mathbb Z}^{d-r}: 2^{n-1}\le |m|<2^{n}\}, \quad n\geq1.
$$
For any $n\geq0$, let  $0\in\Lambda_{m^0_{n}}$ be a spectrum of $\mu_{1}$ satisfies  $\Lambda_{m^0_n}$  is ${8^{N_{n,r}+1}}-$lacunary .
By a reverse triangle inequality,  it implies that 
\begin{equation}\label{min-lacuanary}
\min\{|\lambda-\lambda'|: \lambda\neq\lambda'\in\Lambda_{m_n}\}\geq 4^{N_{n,r}+1}.
\end{equation}
We now enumerate the set 
$$
\{m\in \Z^{d-r}: 2^{n-1}\le |m|<2^{n}\} = \{m_{n,1},...,m_{n,N_{n,r}}\}.
$$
We then take spectrum on each of integers $m_{n,j}$ as follows:
$$
\Lambda_{m_{n,j}}=\Lambda_{m^0_n}+ j\cdot2^{N_{n,r}} {\bf e}_1, \ j=1,...,N_{n,r},
$$
where ${\bf e}_1$ is the unit vector in the $x_1$-direction. For any $m\in\Z^{d-r}\setminus\{0\}$, there is an unique  integer $n\geq0$  such that $2^n\le |m|< 2^{n+1}$. Therefore, by Theorem \ref{product-spectrum},
$$
\Lambda=\bigcup_{m\in\Z^{d-r}}\left(\Lambda_{m}\times\{Am\}\right)=:\small{\left(
  \begin{array}{cc}
   I_r  &  0 \\
    0 & A \\
  \end{array}
\right)}\Lambda'
$$
is a spectrum of $\mu$ where $I_r$ is the $r\times r$ identity matrix and
$$\Lambda'=\bigcup_{m\in\Z^{d-r}}\left(\Lambda_{m}\times\{m\}\right).$$
By Proposition \ref{conjugate}, $\dim^+(\Lambda)=\dim^+(\Lambda')$. Next we will prove that $\dim^+(\Lambda')=0$.

\bigskip

\textbf{Claim 1:}
For any $n\geq0$ and $1\le j_1, j_2\le N_{n,r} $, 
$$\Lambda_{n, m_{j_1}}\cap \Lambda_{n, m_{j_2}}=\emptyset.$$
\proof Suppose there is an positive integer $n$ and $1\le j_1, j_2\le N_{n,r} $ such that 
$$\Lambda_{n, m_{j_1}}\cap \Lambda_{n, m_{j_2}}\neq\emptyset,$$
and choose an element $x$ in the intersection. Then it has two expressions 
$$x=\lambda_1+j_1\cdot2^{N_{n,r}} {\bf e}_1=\lambda_2+j_2\cdot2^{N_{n,r}} {\bf e}_1$$
where $\lambda_1, \lambda_2$ are distinct elements in $\Lambda_{m^0_n}$ as $j_1\neq j_2$. It implies that 
$$\lambda_1-\lambda_2=(j_1-j_2)2^{N_{n,r}} {\bf e}_1.$$
The right hand of the equation implies 
\begin{equation}\label{eq_KK}
    |\lambda_1-\lambda_2|< N_{n,r}\cdot2^{N_{n,r}}\le  2^{N_{n,r}}\cdot 2^{N_{n,r}}={4^{N_{n,r}}}.
\end{equation}
But from \eqref{min-lacuanary}, 
$$|\lambda_1-\lambda_2|\geq {4^{N_{n,r}+1}},$$
which is a contradiction. Hence, the claim is true.

\bigskip

\textbf{Claim 2:}
 $$\delta_n:=\inf\left\{|\lambda_1-\lambda_2|:
 \lambda_1\neq\lambda_2\in\bigcup_{j=1}^{N_{n,r}}\Lambda_{n,j}\right\}=2^{N_{n,r}}.$$
 \proof Note that $2^{N_{n,r}}{\bf e}_1\in\Lambda_{n,1}$, $2\cdot2^{N_{n,r}}{\bf e}_1\in\Lambda_{n,2}$ and
 $$\left|2^{N_{n,r}}{\bf e}_1-2\cdot2^{N_{n,r}}{\bf e}_1\right|=2^{N_{n,r}},$$
 so $\delta_n\le 2^{N_{n,r}}$. On the other hand, for any distinct elements $\lambda_1\in\Lambda_{n, j_1}, \lambda_2\in\Lambda_{n, j_2}$, we can write them as
 $$\lambda_1=\lambda_1'+j_1\cdot2^{N_{n,r}}{\bf e}_1,\quad \lambda_2=\lambda_2'+j_2\cdot2^{N_{n,r}}{\bf e}_1$$
 for some $\lambda_1', \lambda_2'\in\Lambda_{m_n^0}$. If $\lambda_1'=\lambda_2'$, then $j_1\neq j_2$
and
$$|\lambda_1-\lambda_2|=\left|(j_1-j_2)\cdot2^{N_{n,r}}{\bf e}_1\right|\geq 2^{N_{n,r}}.$$
If $\lambda_1'\neq\lambda_2'$, then from \eqref{min-lacuanary}, we have
$$|\lambda_1'-\lambda_2'|\geq 4^{N_{n,r}+1}.$$
Hence, with the inequality obtained in (\ref{eq_KK}),  we have the following 
$$|\lambda_1-\lambda_2|\geq |\lambda_1'-\lambda_2'|-2^{N_{n,r}}\cdot|j_1-j_2| \geq 4^{N_{n,r}+1}-4^{N_{n,r}}>2^{N_{n,r}}.
$$
This justifies the claim. 
 \bigskip

 Now we return to the proof of theorem. For any $h>1$ and $(x_1, x_2)^t\in\R^{r}\times\R^{d-r}$,  choose  $n_1\geq0$ be the biggest integer and $n_2>n_1$ be the smallest  integer such that
$$Q^{d-r}_h(x_2)\cap\Z^{d-r}\subset B(0, 2^{n_2})\setminus B(0, 2^{n_1})\cup\{0\}.$$
Then we have 
$$2^{n_2-1}-2^{n_1+1}\le2h\sqrt{r}\le 2^{n_2}-2^{n_1}.$$
Then 
\begin{equation}\label{eqn2-n1}
n_2-n_1\le 4\log_2(4h\sqrt{r}).
\end{equation}

\bigskip

Note that $Q_h^d((x_1,x_2)^t) = Q_h^r (x_1)\times Q_h^{d-r}(x_2)$. This allows us to have the following:
\begin{eqnarray}\label{no-in-Q}
\#(\Lambda'\cap Q_h^d((x_1, x_2)^t))&=&\sum_{m\in Q_h^{d-r}(x_2)\cap\Z^{d-r}}\#\left(\Lambda_{m}\cap Q_h^{r}(x_1)\right)\nonumber\\
&\le&\sum_{n=n_1+1}^{n_2}\sum_{2^{n-1}\le |m|<2^{n}}\#\left(\Lambda_{m}\cap Q_h^{r}(x_1)\right)+\#\left(\Lambda_{0}\cap Q_h^{r}(x_1)\right).
\end{eqnarray}
We now decompose the first summand into two parts
$$
\sum_{n=n_1+1}^{n_2} (....) = \sum_{\tiny{\begin{array}{l} n_1<n\le n_2,\\  2^{N_{n,r}}>2h\sqrt{r}\end{array}}} (....) +\sum_{\tiny{\begin{array}{l} n_1< n\le n_2,\\  2^{N_{n,r}}\le 2h\sqrt{r}\end{array}}} (....)
$$
In the first case,  
$${2^{N_{n,r}}}>2h\sqrt{r}=\text{diam}(Q^{r}_h(x_1)).$$ 
From the Claim 1 and Claim 2, we have 
$\bigcup\limits_{2^{n-1}\le |m|<2^{n}}\Lambda_{m}$ has at most one element in $Q^{r}_h(x_1)$. So 
\begin{eqnarray}\label{no-1}
\sum\limits_{\tiny{\begin{array}{l} n_1< n\le n_2,\\  2^{N_{n,r}}>2h\sqrt{r}\end{array}}}
\sum\limits_{2^{n-1}\le |m|<2^{n}}   \#\left(\Lambda_{m}\cap Q_h^{r}(x_1)\right)\le&\sum\limits_{\tiny{\begin{array}{l} n_1< n\le n_2,\\  2^{N_{n,r}}>2h\sqrt{r}\end{array}} }1\nonumber\\
\le& n_2-n_1\nonumber\\
\le& 4\log_2(4h\sqrt{r}).
\end{eqnarray}
In the second case, for the integer  $n_1< n\le n_2$ with 
$$2^{N_{n,r}}\le 2h\sqrt{r},$$ 
we have
$$
N_{n,r}\le \log_2 (2h\sqrt{r}).
$$
For $2^{n-1}\le |m|<2^{n}$,   
$\Lambda_m$ is a $8^{N_{n, r}+1}-$lacunary set, so by the Claim in Proposition \ref{no-dim-of-lacu},
\begin{equation*}
\#\left(\Lambda_{m}\cap Q_h^r(x_1)\right)\le\log_{8^{N_{n, r}+1}}(2h\sqrt{r})+2\le  3\log_{2}(2h\sqrt{r}).
\end{equation*}
For $\Lambda_0$, it is a $8-$lacunary set, 
\begin{equation}\label{no-0}
    \#\left(\Lambda_{0}\cap Q_h^{r}(x_1)\right)\le\log_{8}(2h\sqrt{r}) +2\le 3\log_{2}(2h\sqrt{r}).
\end{equation}
We have  
\begin{eqnarray}\label{no-2}
&&\sum_{\{n:n_1< n\le n_2,2^{N_{n,r}}\le2h\sqrt{r}\} }\sum_{2^{n-1}\le |m|<2^{n}}\#\left(\Lambda_{m}\cap Q_h^{r}(x_1)\right)\nonumber\\
& \le & \sum_{\{n:n_1<n\le n_2,2^{N_{n,r}}\le2h\sqrt{r}\} } N_{n,r} (3\log_2(2h\sqrt{d}))\nonumber\\
& \le & (n_2-n_1)\cdot \log_2(2h\sqrt{r})\cdot (3\log_2(2h\sqrt{r}))\nonumber\\
&\le& 12 \cdot \left(\log_2(4h\sqrt{r})\right)^3 \ \ (\mbox{using} \  (\ref{eqn2-n1})).
\end{eqnarray}
Putting \eqref{no-1}, \eqref{no-0}  and \eqref{no-2} into \eqref{no-in-Q}, we have 
$$
\#(\Lambda'\cap Q_h^d((x_1, x_2)^t))\le C\cdot \left(\log_2(4h\sqrt{r})\right)^3
$$
for some constant $C>0$. As for any $\gamma>0$, $$\lim_{h\to +\infty}\frac{\left(\log_2(4h\sqrt{r})\right)^3}{h^{\gamma}} = 0,
$$
we have
$$D^+_{\gamma}(\Lambda')=\limsup_{h\to+\infty}\sup_{(x_1,x_2)\in\R^d}\frac{\#(\Lambda'\cap Q_h^d((x_1, x_2)^t))}{h^{\gamma}}=0.$$
This shows $\dim^+(\Lambda')=0$.
\end{proof}
\medskip
 
\section{Fourier decay}
In this section, we will prove Proposition \ref{prop1.4} and some results about Fourier decay of self-affine measures.

\begin{proof} ({\bf Proof of Proposition \ref{prop1.4}})
As $\mu$ is a singular measure, by \cite[Proposition 2.1]{HLL2013}, we have that the lower ($d$-)Beurling density of $\Lambda$ is zero. i.e. 
$$
D^{-}(\Lambda) = \lim_{h\to\infty} \inf_{x\in{\mathbb R}^d} \frac{\#(\Lambda\cap Q^d_h(x))}{h^d} = 0.
$$
Note that this implies that for all $h>0$, there exists $\xi_h$ such that 
\begin{equation}\label{eq5.1}
Q^d_h(\xi_h)\cap \Lambda = \emptyset.    
\end{equation}
Indeed, if (\ref{eq5.1}) is not true, then one can find $h_0>0$ such that $Q^d_{h_0}(x)\cap \Lambda\ne \emptyset$ for all $x\in\R^d$. But then we partition $\R^d$ into disjoint union of cubes with side length $h_0$ and each cube has at least one element in $\Lambda$. This implies that $D^{-}(\Lambda)\ge h_0^{-1}>0$, a contradiction.

\medskip

Now, using (\ref{eq5.1}), for any $k>0$, we can find $\xi_k$ such that 
\begin{equation}\label{eq1}
Q^d_{2^k}(\xi_k)\cap \Lambda = \emptyset.
\end{equation}
Suppose that $D^{+}_{\alpha}(\Lambda)<\infty$. Then there exists a constant $C'>0$ such that 
$$
\#(\Lambda\cap Q^d_h(x))\le C'h^{\alpha}, \ \forall x\in{\mathbb R}^d.
$$
We now take $f = e^{2\pi i \langle\xi_k,x\rangle}$ into (\ref{eqframelower}), 
$$
\begin{aligned}
A \le& \sum_{\lambda\in \Lambda} |\widehat{\mu}(\xi_k-\lambda)|^2\\
 =& \sum_{\lambda\in \Lambda \cap(Q^d_{2^k}(\xi_k))^{\mathtt C}}|\widehat{\mu}(\xi_k-\lambda)|^2  \ \ \  (\mbox{by (\ref{eq1})})\\
 =& \sum_{j=k}^{\infty}\sum_{\lambda\in Q^d_{2^{j+1}}(\xi_k)\setminus Q^d_{2^j}(\xi_k)}|\widehat{\mu}(\xi_k-\lambda)|^2\\
 \le& C\sum_{j=k}^{\infty}\#(\Lambda\cap (Q^d_{2^{j+1}}(\xi_k)\setminus Q^d_{2^j}(\xi_k))) 2^{-j\gamma} \ \ \ (\mbox{by (\ref{decay condition})})\\
 \le &C' \cdot C \sum_{j=k}^{\infty} (2^{j+1})^{\alpha} 2^{-j\gamma}\\
 = & C'\cdot C \sum_{j=k}^{\infty} 2^{j(\alpha-\gamma)}.
\end{aligned}
$$
Suppose that $\alpha<\gamma$. Then the right hand side above will tend to zero as $k$ tends to infinity which means that $\Lambda$ cannot satisfy \eqref{eqframelower}. Hence, $\alpha \ge \gamma$. In particular, this implies that $\gamma\le \mbox{dim}^+(\Lambda)$ by taking infimum of $\alpha$ such that $D^{+}_{\alpha}(\Lambda)<\infty$.
\end{proof}

\medskip

Because of Proposition \ref{prop1.4} and Theorem \ref{main-thm}, we can show that all self-affine spectral measures we considered have Fourier dimension zero. However, much stronger can be proved easily as follows. 

\begin{proposition}
Let $R\in M_d(\Z)$ be an expansive matrix and $B\in\Z^d$ be a finite set. Suppose that $\#B<|\det (R)|$. Then the Fourier transform  self-affine measure $\mu(R,B)$ does not decay to zero. 
\end{proposition}

\begin{proof}
We think this result is probably well-known. We just present here for completeness. Note that if $\#B<|\det(R)|$. $\mu = \mu(R,B)$ must be singular. By Proposition \ref{key-lemma}, we can find $k\ne0$ and $k\in{\mathbb Z}^d$ such that $\widehat{\mu}(k)\ne 0$.  For all integers $n>0$, noting that $B$ are all integer vectors,
$$
\widehat{\mu}((R^t)^nk) =\prod_{j=1}^{\infty} \widehat{\delta_B}((R^{t})^{-j}((R^t)^nk)) =\prod_{j=n+1}^{\infty}  \widehat{\delta_B}((R^t)^{n-j}k) = \widehat{\mu}(k) \ne0.
$$
This shows that all such self-affine measures do not decay. 
\end{proof}

\end{document}